\documentclass[12pt,a4paper]{article}
\usepackage{amsmath,amssymb,amsthm,amsfonts}
\usepackage{amsbsy}
\usepackage[utf8]{inputenc}
\usepackage{indentfirst}
\usepackage{geometry}
\usepackage{enumitem}
\usepackage[english]{babel}
\usepackage{longtable}
\usepackage{url}
\usepackage[all]{xy}
\usepackage{subcaption}
\usepackage{aliascnt}
\usepackage{ytableau}

\usepackage{verbatim}

\usepackage{tikz}
\usetikzlibrary{matrix}
\usetikzlibrary{arrows}
\usetikzlibrary{fit}
\usetikzlibrary{shapes}

\usepackage{a4}
\usepackage{setspace}
\usepackage{makeidx}
\usepackage{multirow}
\usepackage{epsfig}
\usepackage{float}
\usepackage{graphicx}
\usepackage{color}
\usepackage{caption}
\usepackage{multicol}
\usepackage[lined,boxed,commentsnumbered]{algorithm2e}
\newcommand{\para}{\par\vspace{.20cm}}
\newcommand{\cut}{\textsf{cut}}

\newtheorem{theorem}{Theorem}

\newtheorem{lemma}[theorem]{Lemma}
\newtheorem{cor}[theorem]{Corollary}

\begin{document}
	\baselineskip 18pt
	
	\title{\bf Group rings and the RS-property}
	\author{Gurmeet K. Bakshi  \\ {\em \small Centre for Advanced Study in
			Mathematics,}\\
		{\em \small Panjab University, Chandigarh 160014, India.}\\{\em
			\small email: gkbakshi@pu.ac.in} \and  Sugandha Maheshwary {\footnote {Research supported by DST, India (INSPIRE/04/2017/000897).} \footnote{Corresponding author}}
		\\ {\em \small Indian Institute of Science Education and Research, Mohali,}\\
		{\em \small Sector 81, Mohali (Punjab)-140306, India.}
		\\{\em \small email: sugandha@iisermohali.ac.in}
		\and Inder Bir S. Passi \\  {\em \small Centre for Advanced Study
			in Mathematics,}\\ {\em \small Panjab University, Chandigarh-160014, India} \\
		{\small \& }\\
		{\em \small Indian Institute of Science Education and Research, Mohali,}\\
		{\em \small Sector 81, Mohali (Punjab)-140306, India.}\\{\em \small email: ibspassi@yahoo.co.in } }
	\date{}
	{\maketitle}
	
	\begin{abstract}\noindent {}
		The object of this paper is to  study (infinite)  groups whose integral group rings have only trivial central units. This property is closely related to a property, here called the RS-property (\cite{DMS05}, \cite{RS90}),  involving  conjugacy in the group. 
	\end{abstract}\vspace{.25cm}
	{\bf Keywords} : integral group rings, unit group, trivial central units, FC-subgroup. \vspace{.25cm} \\
	{\bf MSC2000 :} 16U60; 16S34; 20C07; 20E06; 20E45;  20F14

	\section{Introduction}Given a group $G$, let $\mathcal{U}(\mathbb{Z}[G])$ be the group of  units of the integral group ring $\mathbb{Z}[G]$ and let $\mathcal{Z}(\mathcal U(\mathbb Z[G]))$ be its center. Trivially, $\mathcal{Z}(\mathcal U(\mathbb Z[G]))$ contains $\pm \mathcal Z(G)$, where $\mathcal{Z}(G)$ denotes the center of $G$. In case  $\mathcal{Z}(\mathcal U(\mathbb Z[G]))= \pm \mathcal Z(G)$, i.e., all central units of $\mathbb{Z}[G]$ are trivial, following \cite{BMP17}, we call $G$ a \cut-{\it group} or a group with the \cut-{\it property}. 
	The question of classifying {\sffamily cut}-groups was explicitly posed, for the first time, by Goodaire and Parmenter \cite{GP86}. As an answer, Ritter and Sehgal \cite{RS90}  gave a characterization for finite \cut-groups which was later generalized by Dokuchaev, Polcino Milies  and Sehgal \cite{DMS05} to arbitrary groups. Let us say that an element $x\in G$ of finite order has the {\it RS-property}  (or is an {\it RS-element}) in $G$ if
	\begin{equation}\label{E2}
	x^{j}\sim_{G}x^{\pm 1} \ \text{for all} ~j\in U(o(x)),
	\end{equation} 
	where $o(x)$ denotes the order of $x$, $U(n):=\{j: 1\leq j \leq n, ~gcd(j,\,n)=1\}$, and $y\sim_G z$ denotes $y$ is conjugate to $z$ in $G$.  
	Let $\Phi(G)$ denote the FC-subgroup of $G$, i.e., the subgroup consisting of those elements of $G$ which have only finitely many conjugates in $G$, and $\Phi^+(G)$ its torsion subgroup. Then the  characterization of {\sf cut}-groups given in \cite{DMS05} can be stated to say that $G$ is a {\sf cut}-group if, and only if, every element of $\Phi^+(G)$ has the RS-property in $G$. Recently, finite {\sf cut}-groups and their properties have been explored further (\cite{Bac17}-\cite{CD10},\,\cite{Mah18}; see also \cite{MP18}). Our purpose in the present work is to   examine the class of infinite {\sf cut}-groups. This class is neither subgroup-closed nor is it quotient-closed. 
	\para In Section 2, we show (Lemma \ref{P1}) that if $A$ is an RS-subgroup of a group $G$, then $A/N$ is RS-subgroup of $G/N$ for every finite normal subgroup $N$ of $G$ contained in $A$ and thus deduce (Theorem \ref{P0})  that the class of \cut-groups is closed under quotients by finite normal subgroups. We examine the class of \cut-groups under extensions (Theorem \ref{P2}) including amalgams and HNN extensions and thus provide several interesting examples of infinite \cut-groups. Extending the result on characterization of finite metacyclic \cut-groups (\cite{BMP17}, Theorem 5),  a classification (Theorem \ref{TM}) of infinite metacyclic \cut-groups  has been given. We also classify $p$-groups (Theorem \ref{p2}) and nilpotent groups (Theorem \ref{p3}) which are \cut-groups.  For a finite group $G$ and a normal subgroup $A$ of $G$, it has been  shown (Theorem \ref{L0}) that $ \mathcal Z(\mathcal U(\mathbb Z[G]))\cap  (1+ \Delta(G)\Delta(A))$ is trivial, if and only if,  the rank of  $\mathcal Z(\mathcal U(\mathbb Z[G]))$ equals that of $\mathcal Z(\mathcal U(\mathbb Z[G/A]))$, where $\Delta(G)$ denotes the augmentation ideal of $\mathbb{Z}[G]$.  Given  a finite normal subgroup $A$ of a solvable group $G$,  $G/C_{G}(A)$, where $C_{G}(A)$ denotes the centralizer of $A$ in $G$, is a finite solvable group. Thus if $G/C_{G}(A)$ is a \cut-group, then,  by (\cite{Bac17}, Theorem 1.2),  no prime other that $2$, $3$, $5$ and $7$ can divide the order of $G/C_{G}(A)$. If, in addition,  $ \mathcal Z(\mathcal U(\mathbb Z[H]))\cap  (1+ \Delta(H)\Delta(A)) $ is trivial,  then we show (Corollary \ref{c0}) that the order of $A$ is also not divisible by any prime different from $2$, $3$, $5$ and $7$,  where $H:=A\rtimes G/C_{G}(A)$.
	\para In Section 3, we consider symmetric central units; a unit $\sum u_{g}g \in \mathbb{Z}[G]$ being  symmetric if  $\sum u_{g}g = \sum u_{g}g^{-1}$.  We prove (Theorem \ref{T3}) that, modulo the trivial units, the group of central units is isomorphic to a torsion free subgroup of $\mathcal{Z}(\mathcal U(\mathbb Z[G]))$ consisting of symmetric central units. Consequently, it follows that all central units of $\mathbb{Z}[G]$ are trivial, if so are all the symmetric central units. \para Finally, in Section 4, we observe (Theorem \ref{T0}) that $\mathcal{Z}_{i}(\mathcal{U})/\mathcal{Z}_{i-1}(\mathcal{U})$ is of finite exponent for all $i \geq 2$, provided  $\mathcal{Z}(G)$ is of finite exponent or $G$ is generated by torsion elements of bounded exponent,   where $\mathcal{Z}_{i}(\mathcal{U})$ denotes the $i$th term of the upper central series of $\mathcal{U}(\mathbb{Z}[G])$. 

\section{\cut-groups}
Let $G$ be an arbitrary group. Let us  say that a  subgroup $A$ of $G$ is an {\it RS-subgroup} of $G$, if every torsion element of $A$ is an RS-element in $G$. Such subgroups are relevant for the study of \cut-groups because of the following characterization:  \begin{quote} {\it If $A$ is a normal subgroup of a finite group $G$, then 
$A$ is an \linebreak RS-subgroup of $G$ if, and only if, $\mathcal Z(\mathcal U(\mathbb Z[G]))\cap \mathbb Z[A]$ consists of trivial units (\cite{DMS05}, Theorems 8 \& 9). } \end{quote} Furthermore, a key step in the reduction of the investigation of infinite \cut-groups to that of finite groups is provided by the following:

\begin{lemma}\label{e0}(\cite{DMS05}, Lemma 10) If $G$ is an arbitrary group and $A$  a finite normal subgroup of $G$,  then there exists a finite extension  $H$ of $A$ such that  $$\mathcal{Z}(\mathcal{U}(\mathbb{Z}[G]))\cap\mathbb{Z}[A]=\mathbb{Z}(\mathcal{U}(\mathcal{Z}[H]))\cap\mathbb{Z}[A].$$\end{lemma}

The various known criteria, available in  \cite{DMS05}, \cite{JJMR02} and  \cite{RS90}  for a group to be a  \cut-group can be put together as follows:

\begin{theorem}\label{p5}
	For every group $G$,  the following statements are equivalent:
	\begin{description}
		\item[(i)] $G$ is a {\sf cut}-group.
		\item[(ii)] $\Phi^{+}(G)$ is an RS-subgroup of $G$.
		\item[(iii)] $\mathcal{Z}(\mathcal{U}(\mathbb{Z}[G]))\cap \mathbb{Z}[\Phi^{+}(G)] = \pm T(\mathcal Z(G))$, where  $T(\mathcal Z(G))$ denotes the torsion subgroup of  $\mathcal Z(G)$.
		\item[(iv)] $\pm G = \mathcal{N}_{ \mathcal{U}}(G)$, the normalizer of $G$ in $\mathcal{U}:= \mathcal{U}(\mathbb{Z}[G]).$	\end{description}
\end{theorem}

Clearly, the property of being an RS-subgroup is   subgroup-closed. This property  also turns out to be closed under taking quotients by finite normal subgroups. To be precise, we have 
\begin{lemma}\label{P1}
	Let $G$ be a group and $A$ a subgroup of $G$. Let $N$ be a finite normal subgroup of $G$ contained in $A$. If $A$ is an RS-subgroup of $G$, then $A/N$ is an RS-subgroup of $G/N$.  
\end{lemma}
\begin{proof}
	Let $\overline{a}:= aN\in A/N$ be a torsion element. We need to check that for all $j \in U(o(\overline{a}))$, $\overline{a}^{j}\sim_{G/N}\overline{a}^{\pm1}$. Since $N$ is finite, every element of the coset $aN$ is a torsion element. Let $a_{0}$ be an element of minimal order in the coset $aN$. Let $i$ be such that $ij \equiv 1 \mod o(\overline{a}) $, so that $\overline{a}_{0}^{ij} = \overline{a}_{0}$,  and consequently, by the minimality of the order of $a_{0} $, $o(a_{0}^{ij})=o(a_{0})$. This gives  $\gcd(ij,\,o(a_{0}))=1$ and hence $\gcd(j,o(a_{0}))=1$. Since by assumption $A$ is an RS-subgroup of $G$, we have $a_{0}^{j}\sim_{G}a_{0}^{\pm1}$, which yields  $(\overline{a})^{j}\sim_{G/N}\overline{a}^{\pm1}$. 
\end{proof}

If $G$ is a finite {\sf cut}-group, then so is $G/N$ for every $N\unlhd G$ \cite{RS90}. As a generalization of this fact, we have  the following:

\begin{theorem}\label{P0}
		If $G$ is a {\sf cut}-group, then $G/N$ is a {\sf cut}-group for every finite normal subgroup $N$ of $G$.
\end{theorem}

\begin{proof}
	Let $A/N$ be a finite normal subgroup of $G/N$. Since $N$ is finite,  $A$ is a finite normal subgroup of  $G$ and thus it is an RS-subgroup of $G$. By Lemma \ref{P1}, $A/N$ is an RS-subgroup of $G/N$. Consequently, $\Phi^{+}(G/N)$ is an RS -subgroup of $G/N$. Hence Theorem \ref{p5} yields that   $G/N$ is a \cut-group.
\end{proof}

We next consider the behaviour of \cut-groups under extensions. Note that, in view of Theorem \ref{p5},  every group $G$ with $\Phi^{+}(G) = \{1\}$ is a \cut-group; in particular, a torsion-free group is a \cut-group.  
 
\begin{theorem}\label{P2}  
	\begin{description}
\item[(i)] Let $G$ be a normal subgroup of the group $\Pi$ and $Q=\Pi/G$. 
	\begin{description}
		\item[(a)]
		If  $Q$ is a  \cut-group, and $G\cap \Phi^{+}(\Pi)= \{1\}$ (in particular, if   $\Phi^+(G)=\{1\}$),  then $\Pi$ is a \cut-group.
		\item[(b)] If $\Phi^+(Q)=\{1\}$,  and $\Phi^+(G)$ is an RS-subgroup of $\Pi$ (in particular, if   $G$ is a \cut-group),  then $\Pi$ is a \cut-group. 
			\end{description}
	\item[(ii)] If  $\Pi= G \ast_{A} G'  $ is an amalgam  of arbitrary groups $G$ and $G'$ with the amalgamated subgroup $A$ an RS-subgroup  of $G$ or $G'$, then  $G$ is a \cut-group, provided  $A \neq G$ and $A \neq G'$.   In particular, the free product of arbitrary non-trivial groups is a \cut-group. 
	 \item[(iii)] If $\Pi$ is an HNN extension of a group $G$ over isomorphic subgroups $A$ and $B$ such that one of $A$ or $B$ is an RS-subgroup of $G$, then $\Pi$ is a \cut-group. 	\end{description} 
\end{theorem}  
\begin{proof} (i) Let $f:\Pi\longrightarrow Q$ be the canonical epimorphism and let $Q$  be  a  \cut-group. Let $H$ be a finite normal subgroup of $\Pi$, so that $f(H)$  is a finite normal subgroup of the \cut-group $Q$. Let $h \in H$ and $j$  a positive integer coprime to $o(h)$. Clearly then  $j$ is relatively prime to $o(f(h))$, and hence $${f(h)}^{j}\sim_{Q}{f(h)}^{\pm 1},$$ i.e., there exists $q\in Q$ such that 
\begin{equation}\label{e1}
q^{-1}{f(h)}^{j}q={f(h)}^{\pm 1}.
\end{equation}
Let $y\in\Pi$ be such that $f(y)=q$. In view of Eq.\,(\ref{e1}), we have $f(y^{-1}h^jyh^{\mp1})=1$. Observe that $y^{-1}h^jyh^{\mp 1}\in G\cap \Phi^{+}(\Pi)$. Now, if $G\cap \Phi^{+}(\Pi)= \{1\}$,  then $y^{-1}h^jyh^{\mp 1}=1$, i.e., $h^{j}\sim_{\Pi} h^{\pm 1}$. Hence $H$ is an RS-subgroup of $\Pi$  and consequently, $\Pi$ is a \cut-group. This proves (a). 
Next consider  $\Phi^+(Q)=\{1\}$.  Then every finite normal subgroup of $\Pi$ is necessarily  contained in $G$. Therefore, $\Phi^+(\Pi)$ is contained in $G$, and hence in $\Phi^+(G)$. Now if, $\Phi^+(G)$ is an RS-subgroup of $\Pi$, then it follows that $\Phi^+(\Pi)$ is an RS-subgroup of $\Pi$ and therefore $\Pi$ is a \cut-group. Finally, we can see that if $G$ is a \cut-group, then $\Phi^+(G)$ is an RS-subgroup of $G$ and therefore also of $\Pi$. This proves (b). \vspace{.3cm} \\ (ii) 	Let  $\Pi= G \ast_{A} G'  $ be an amalgam  of arbitrary groups $G$ and $G'$ with the amalgamated subgroup $A$  an RS-subgroup  of $G$ or $G'$. Suppose that  $A \neq G$ and $A \neq G'$.   \para If the index of $A$ in both $G$ and $G'$ is two, then $\Pi/A$ is isomorphic to the  free product $C_{2} \ast C_{2}$ and hence $\Phi^+(\Pi/A)=\{1\}$. Also $A$ an RS-subgroup of one of $G$ or $G'$ implies that  $A$ is an RS-subgroup of $\Pi$. Consequently $\Phi^+(A)$ is an RS-subgroup of $\Pi$. Thus (i)(b) yields that $\Pi$ is a \cut-group. \para If the index of $A$ is atleast  three in either $G$ or  $G'$, then, by (\cite{Cor09}, Proposition 1),  $\Phi^{+}(\Pi)$ is a subgroup of  $A$.  Since $A$ is an RS-subgroup of either $G$ or $G'$, $A$ is also an RS-subgroup of $\Pi$.  Consequently,  $\Phi^{+}(\Pi)$ is an RS-subgroup of $\Pi$, and hence  $\Pi$ is a \cut-group. \vspace{.3cm}  \\ (iii) Let $\Pi$ be  an HNN extension of a group $G$ over isomorphic subgroups   $A$ and $B$ of $G$. If either  $A$ or $B$ is a proper subgroup of $G$, then, by (\cite{Cor09}, Proposition 3), $\Phi^{+}(\Pi)$ is a subgroup of both $A$ and $B$.  Since one of $A$ or $B$ is an RS-subgroup of $G$, it follows that $\Phi^{+}(\Pi)$ is an RS-subgroup of $G$ and, consequently, that of $\Pi$, as desired.   If $A=B=G$, then $G$ is a \cut-group and $\Pi/G$ is infinite cyclic. Hence,  by (i)(a),  $\Pi$ is a \cut-group. 
\end{proof} 
The above theorem enables us to construct interesting examples of \cut-groups. \vspace{.3cm}\\
{\bf Examples}  \begin{description}
	\item[1.] $~$ While not every finite simple group is a \cut-group (\cite{AKS08}, \cite{Fer04}, also see \cite{MP18}, Theorem 2), observe  that if $G$ is an  infinite simple group, then clearly  $\Phi^{+}(G) = \{1\}$ and therefore it is a \cut-group. Furthermore,  as an immediate consequence of Theorem \ref{P2}(i)(a), it follows that \begin{quote} {\it An extension of an infinite simple group by  a \cut-group is a \cut -group.} \end{quote} 
\item[2. ] $~$ Recall that  $PSL(n, k)$ is simple if  $k$ is a field of characteristic $0$ and $n \geq 2$. If the roots of unity in $k$ are of exponent dividing $4$ or $6$ (e.g., if $k=\mathbb{Q}$ or $\mathbb{R}$), then  Theorem \ref{P2}(i)(b) yields that \begin{quote}
	{\it $SL(n, k)$ is a \cut-group for $ n \geq 2$.}  
\end{quote}
 \item[3.] $~$ It is known that the modular group $PSL(2, \mathbb{Z})$  is the  free  product of  cyclic groups $C_{2} $ and $ C_{3}$.  Thus, Theorem \ref{P2}(ii)  yields that 
\begin{quote} {\it The modular group $PSL(2, \mathbb{Z})$ is a \cut-group. } \end{quote}  
\item[4.]  $~$ Observe that   $SL(2, \mathbb{Z})$ is isomorphic to $ C_{4} \ast_{C_{2}} C_{6}$ and thus, by Theorem \ref{P2}(ii),  \begin{quote}
{\it 	$SL(2, \mathbb{Z})$ is a \cut-group. } \end{quote}
\item[5.]  $~$ The Baumslang Solitar group $BS(m, n):=\langle a, t~|~t^{-1}a^{m}t=a^n\rangle $,  where $m$ and $n$ are  non-zero integers,  is an  HNN extension. Thus, by Theorem \ref{P2}(iii)\begin{quote}
	{\it The Baumslag Solitar groups $BS(m, n)$ are  \cut-groups.}
\end{quote} 
 \end{description}

\para We now proceed to show that Theorem \ref{P2} enables us to classify infinite metacyclic \cut-groups. It may be mentioned that a complete list (up to isomorphism) of finite metacyclic \cut-groups has been computed in \cite{BMP17}. For  every  group $G$, $\Phi^+(\mathcal{Z}(G))\subseteq\Phi^+(G)$. If the equality holds, then we can see that $G$ is \cut-group if, and only if, $\mathcal{Z}(G)$ is a \cut-group, i.e., each central torsion element must have order dividing $ 4$ or $ 6$. This observation is helpful for the study  of infinite metacyclic \cut-groups. 		
\begin{theorem}\label{TM}An infinite non-abelian metacyclic group $G$ is a \cut-group if, and only if, it is isomorphic to one of the following groups:
	
	\begin{description}
		\item[(i)] $ \langle a,~b~|~b^{n}=1, ba=a^{-1}b \rangle,~ n \in  \{0, 2,4,6,8,12\}; $
		\item[(ii)]  $ \langle a,~b~|~a^{m} =1,ba= a^{r}b\rangle,~m\geq 3,~1\neq r\in U(m)$ and $U(m)=\langle -1, r\rangle.$
	\end{description}
\end{theorem}
\begin{proof} Let $G$ be an infinite metacyclic group and $N= \langle a \rangle $  a cyclic  normal subgroup of $G$ with $G/ N= \langle bN \rangle $ cyclic. Let $m= \operatorname{o}(a)$ and $n= \operatorname{o}(b)$. As $G$ is infinite, one of $m$ or $n$ must be $0$. 
	\par\vspace{.25cm}
	
	\underline{\textbf{Case I: $m=0$.}}\\
	 In this case, $$G\cong \langle a,~b~|~b^{n}=1,ba=a^{-1}b\rangle.$$ If $n=0$, then, by Theorem \ref{P2}(i), $G$ is a \cut-group. \para We thus assume that $ n \neq 0$. Observe that $\mathcal{Z}(G)= \langle b^{2} \rangle $. We assert  that $\Phi^+(G)= \Phi^+(\mathcal{Z}(G))= \langle b^{2} \rangle $.  Let $A$ be a finite normal subgroup of $G$ and let $g\in A$, so that $g=a^{\alpha}b^{\beta}$ for some $\alpha, \beta \in \mathbb{Z}$. Now, $g^{b}=a^{-\alpha}b^{\beta}=a^{-2\alpha}g\in A$ implies $\alpha=0$, as $A$ is finite. Furthermore, $g^{a}=a^{(-1)+(-1)^{\beta}}b^{\beta}g\in A$ implies $2\mid \beta$ and hence $A\subseteq \langle b^{2} \rangle$. Consequently, $\Phi^+(G) \subseteq  \langle b^{2} \rangle $; however, the reverse inclusion is obvious. Thus the assertion follows. Consequently,  by the foregoing observation,  $G$ is a \cut-group if, and only if, $n \in  \{2,4,6,8,12\}.$
	 	\par\vspace{.25cm}
	 
	 \underline{\textbf{Case II: $m \neq  0$.}}\\	
	 	In this case, $n=0$ and so 
	 $$G\cong \langle a,~b~|~a^{m} =1,ba= a^{r}b\rangle,~r\in U(m).$$  Since $G$ is non-abelian,  $m \geq 3$. By Theorem \ref{P2}(i)(b), $G$ is a \cut-group if, and only if, $a$ is an RS-element of $G$.  It is easy to see  that $a$ is an RS-element of $G$ if, and only if,   $U(m)=\langle -1, r\rangle$. 
\end{proof}

 \para We next  consider $p$-groups.
  
 \begin{theorem}\label{L2} A non-trivial normal subgroup $A$ of a finite $p$-group $G$ is an RS-subgroup of $G$ if, and only if, one of the following holds:
 	\begin{description}
 			\item[(i)] $p=2$ and $a^3\sim_{G} a^{\pm 1}$ for all $a \in G$;
 		\item[(ii)] $p=3$ and  $a^2\sim_{G} a^{-1}$ for all $a \in G$ .
 	
 	\end{description}
 	
 \end{theorem}
 
 \begin{proof}
 	A non-trivial normal subgroup $A$ of $G$ must  intersect $\mathcal Z(G)$ non-trivially, as $G$ is a $p$-group. However, a central element of $G$  is an RS-element  if, and only if, its order divides  $4$ or $6$.  This gives that $ p = 2$ or $3$. 
 	\par\vspace{.25cm}
 	 	 Let $1\neq a\in A$. If $p=2$, then $a^{3}\sim_{G}a^{\pm1}$, as $3\in  U(o(a))$ and for the same reason, if $p=3$, then $a^{2}\sim_{G}a^{\pm1}$. In the latter case, we just need to check that $a^{2}\not\sim_{G}a$. Consider the lower central series $\{\gamma_{i}(G)\}_{\geq 0}$  of $G$.  Since $G$ is a $p$-group, there exists $n \geq 1$ such that $\gamma_{n}(G)=\{1\} $. Hence we can  find  $1 \leq i <n$  such that $a\in \gamma_{i}(G)\setminus\gamma_{i+1}(G)$. If $a^{2}\sim_{G}a$, then $a^{2}= g^{-1}ag$  for some $g\in G$, which implies $a=[a,\,g]\in [\gamma_{i}(G), G]=\gamma_{i+1}(G)$, contradicting the choice of $i$.
 	\par\vspace{.25cm}
 	Conversely, let $G$ be a finite $p$-group, $p\in \{2,\,3\}$, and $A$ a normal subgroup of $G$ satisfying (i) or (ii) according as $p=2$ or $3$. 
 	\par\vspace{.25cm}
 	Let $a\in A$. 
 
 	\par\vspace{.25cm}
 	 If  $p=2$, then $U(o(a))=\langle 3\rangle \times \langle -1\rangle $ and hence $a^{3}\sim_{G}a^{\pm 1}$ which implies $ a^{j}\sim_{G}a^{\pm 1}$ for all $j \in U(o(a))$.
 	 	\par\vspace{.25cm}
 	 If $p=3$, then $2$ is a primitive root modulo $o(a)$ and hence $a^{2}\sim_{G}a^{-1}$ implies that $ a^{j}\sim_{G}a^{-1}$ for all $ j \in U(o(a))$.
 	 \par\vspace{.25cm}
 	  Therefore, in either case, $a$ is an RS-element of $G$.  \end{proof} 
 The above result yields information about the RS-subgroups of a finite nilpotent group. Given a group $G$, let $\pi(G)$ denote the set of primes $p$ for which $G$ contains an element of order $p$.
 \begin{cor}\label{L1} Let $A$ be a normal RS-subgroup of a finite nilpotent group $G$, then $\pi(A)\subseteq\{2,3\}$.
 \end{cor}
 \begin{proof} Observe that  for every prime $p\in \pi(G)$,   the  Sylow $p$-subgroup of $A$ is an RS-subgroup of the Sylow $p$-subgroup of $G$. Thus the assertion follows from Theorem   \ref{L2}.\end{proof}

In view of  Lemma \ref{e0},  Theorems \ref{p5} \& \ref{L2},   we obtain the following:

\begin{theorem}\label{p2}  A $p$-group $G$ is a \cut-group if, and only if, one of the following holds:
	
	\begin{description}
			\item[(i)] $p=2$ and $a^3\sim_{G} a^{\pm 1}$ for all $a\in \Phi^+(G)$ ;
		\item[(ii)] $p=3$ and $a^2\sim_{G} a^{-1}$ for all $a\in \Phi^+(G)$ ;
		\
		\item[(iii)] $\Phi^+(G)=\{1\}$.
	\end{description}
	
\end{theorem}
\par\vspace{.25cm}
	 It may  be noted that, while a finite $p$-group, which is a  \cut-group, must necessarily be a 2-group or a 3-group, this is not the case for  an  infinite $p$-group to be a \cut-group.  For example, for any prime $p$, the wreath product $ C_{p} \wr A$, where $A$ is a direct sum of infinitely many copies of $C_{p}$, is a \cut-group, since  $\Phi( C_{p} \wr A)=\{1\}$. For more general examples of groups with  $\Phi^{+}(G) = \{1\}$, see  \cite{Pre13}-\cite{Pre06}. 

\par\vspace{.25cm} Let $G_p$ denote the subset of $G$ consisting of $p$-elements in $G$.  
For  nilpotent  \cut-groups, we have 

\begin{theorem}\label{p3} A nilpotent group $G$ is a \cut-group if, and only if, either\linebreak  $\Phi^+(G)=\{1\}$ or $\Phi^+(G)$ is a $\{2,\,3\}$-group and the following conditions hold:
	
	\begin{description}
		
		\item[(i)] for all $a\in \Phi^+(G)_2$, $a^3\sim_{G} a^{\pm1}$;
		\item[(ii)] for all $a\in \Phi^+(G)_3$, $a^2\sim_{G} a^{-1}$. 
	\end{description}	
\end{theorem}
\begin{proof} 
Let $G$ be a nilpotent group with  $\Phi^+(G) \neq \{1\}$. Let $x \in \Phi^+(G)$ be an element of prime order (say $p$) and $A$ a finite normal subgroup of $G$ containing $x$. Since $G$ is a \cut-group, $A$ is an RS-subgroup of $G$. Consider $H:= A\rtimes G/C_{G}(A)$. Observe that $A$ is an RS-subgroup of the finite nilpotent group $H$.  Thus, by Corollary \ref{L1}, $\pi(A) \subseteq \{2, 3\}$, which gives that $p=2$~or~$3$. Furthermore, as in Theorem \ref{L2},  it turns out  that $a^3\sim_{G} a^{\pm 1}$  if $p=2$,  and $a^2\sim_{G} a^{-1}$ if $p=3$.  The converse can be seen easily as $U(2^i)=   \langle 3\rangle \times \langle -1\rangle $, $U(3^j)=  \langle 2\rangle $  and $U(2^i3^j)= U(2^i) \times U(3^j)$ for all $i, j \geq 1$. 
\end{proof}

\begin{theorem}\label{L0} Given a normal subgroup $A$ of a finite group $G$, the following statements are equivalent:
	\begin{description}\item[(i)] $ \mathcal Z(\mathcal U(\mathbb Z[G]))\cap   (1+ \Delta(G)\Delta(A)) $ is trivial.
		\item[(ii)] $\rho(G)=\rho(G/A),$  where $\rho(G)$ denotes the rank of  $\mathcal Z(\mathcal U(\mathbb Z[G]))$. 
	
	\end{description}
	 \end{theorem}
\begin{proof}
 Suppose $(i)$ holds. Let $$\mathbb{Q}[G] \cong \oplus_{1\leq i \leq m} M_{n_{i}}(\mathbb{D}_{i})$$ be the Wedderburn decomposition of $\mathbb{Q}[G]$. By reordering, if necessary, we can assume that $$\mathbb{Q}[G] (1-\hat{A})\cong \oplus_{1\leq i < r} M_{n_{i}}(\mathbb{D}_{i})$$ and $$\mathbb{Q}[G/A] \cong \mathbb{Q}[G] \hat{A} \cong \oplus_{r \leq i \leq m} M_{n_{i}}(\mathbb{D}_{i}),$$ where $1 \leq r < m$. The center $\mathcal{Z}(\mathbb{D}_{i})$ of each division ring $\mathbb{D}_{i}$ is a finite extension of $\mathbb{Q}$.  Let $\mathcal{O}_{i}$ be the ring of integers of $\mathcal{Z}(\mathbb{D}_{i})$. We then have $$\rho(G)= \sum_{1 \leq i \leq m} \rho(\mathcal{U}(\mathcal{O}_{i}))$$ and  $$\rho(G/A)= \sum_{r \leq i \leq m} \rho(\mathcal{U}(\mathcal{O}_{i})).$$  Consequently,   $\rho(G/A) \leq \rho(G)$. Thus, to establish  (ii), it suffices to prove that under the natural map $\pi:\mathbb{Z}[G] \rightarrow \mathbb{Z}[G/A]$, any set of linearly independent central units in $\mathbb{Z}[G]$  are mapped to linearly independent central units of $\mathbb{Z}[G/A]$. 
 \par\vspace{.25cm}
 Let $u_{1}, u_{2}, \cdots, u_{t}$ be linearly independent central units in $\mathbb{Z}[G]$.   Denote $\pi(u_{i})$ by  $\overline{u}_{i}$. Suppose $\overline{u}^{k_{1}}_{1}\overline{u}^{k_{2}}_{2} \cdots \overline{u}^{k_{t}}_{t}=\overline{1}$ for some integers $k_{1}, k_{2}, \cdots, k_{t}$. Then $$u^{k_{1}}_{1}u^{k_{2}}_{2}\cdots u^{k_{t}}_{t}-1 \in \Delta(G, A).$$ So $u= u^{k_{1}}_{1}u^{k_{2}}_{2}\cdots u^{k_{t}}_{t} \in 1+ \Delta(G, A)$. Consequently, $ u \equiv a \mod \Delta(G)\Delta(A)$ for some $a \in A$. This gives $ua^{-1} \equiv 1  \mod \Delta(G)\Delta(A)$. Since $G$ is finite and $u$ is central, it follows that  $u^m$ belongs to $1 + \Delta(G)\Delta(A)$ for some $m \geq 1$. Therefore, $u^m$ and hence $u$ is trivial, i.e.,  $u= u^{k_{1}}_{1}u^{k_{2}}_{2}\cdots u^{k_{t}}_{t} \in \pm \mathcal{Z}(G)$.  Raising it to a suitable power and  using the fact that $u_{1}, u_{2}, \cdots, u_{t}$ are linearly independent, it follows that $k_{i}=0$ for all $i$, thus  proving the validity of (ii).
 
 \par\vspace{.25cm}  Conversely, if (ii) holds, then it follows that under the natural map $\pi:\mathbb{Z}[G] \rightarrow \mathbb{Z}[G/A]$, any set of linearly independent central units in $\mathbb{Z}[G]$  are mapped to linearly independent central units of $\mathbb{Z}[G/A]$. Suppose $$u = \sum_{a \in A} u_{a} a \in \mathcal{Z}(\mathcal{U}(\mathbb{Z}[G])) \cap   (1+ \Delta(G)\Delta(A)).$$ Then $\pi(u) =1$ and hence   $\{u\}$  is mapped to a linearly dependent set.  Thus $\{u\}$ is linearly dependent and so  $u$ can't have infinite order.   Consequently $u$ is a trivial central unit and  (i) holds.
\end{proof}

\begin{cor}\label{c0}
	Let $A$ be a finite normal subgroup of a solvable group $G$ and let  
	$H:=A\rtimes G/C_{G}(A)$. Suppose   that  \begin{description}
		\item[(i)] $ \mathcal Z(\mathcal U(\mathbb Z[H]))\cap  (1+ \Delta(H)\Delta(A)) $ is trivial;
		\item[(ii)] $G/C_{G}(A)$ is a \cut-group.
	\end{description}
	Then,
	$\pi(A)\subseteq \{2,3,5,7\}$. \end{cor}
\begin{proof}
	Observe  that $H$ is finite. Since $H/A\cong G/C_{G}(A)$ and $G/C_{G}(A)$ is a \cut-group,  $\rho(H/A)=\rho(G/C_{G}(A))=0$. In view of (i), Theorem \ref{L0} yields that   $\rho(H)=\rho(H/A)$. Thus $\rho(H)=0$, i.e., $H$ is a finite solvable \cut-group. Consequently,  by (\cite{Bac17}, Theorem 1.2) ,  $\pi(H)\subseteq \{2,3,5,7\}$, which gives the desired result.  
\end{proof}

\section{Symmetric central units }
\vspace{.5cm} \par Given $u= \sum u_{g}g \in \mathbb{Z}[G]$, let $u^{\ast}= \sum u_{g}g^{-1}$. An element $ u \in \mathbb{Z}[G]$ is called symmetric, if $u = u^{\ast} $.  \para It is known (see \cite{JdR15}, Corollary 7.1.9) that if $G$ is a finite abelian group, then $\mathcal{U}(\mathbb{Z}[G])$ is the direct product of $\pm G$ and a torsion free group of symmetric units. As a generalization of this result, we have the following:  
\begin{theorem}\label{T3} For every  group $G$, the following statements hold:
	\begin{description} 
		\item[(i)] There is an exact sequence $1 \rightarrow \pm \mathcal{Z}(G) \rightarrow \mathcal{Z}(\mathcal{U}(\mathbb{Z}[G])) \rightarrow \mathcal{S}\rightarrow 1$, where $\mathcal{S}$ is a torsion free subgroup of $\mathcal{Z}(\mathcal{U}(\mathbb{Z}[G]))$  consisting of symmetric central units.  
		\item[(ii)] If $u\in \mathcal{Z}(\mathcal{U}(\mathbb{Z}[G])),$ then there exists an element  $g \in \mathcal{Z}(G)$ such that $u=gu^{\ast}$. Furthermore, $u^{2} \in  \mathcal{Z}(G)\mathcal{S}$.  \item[(iii)] If the symmetric central units of $\mathbb{Z}[G]$ are trivial, then so are all central units, i.e., $G$ is a \cut-group. \end{description} \end{theorem}
\begin{proof} (i) Observe that the map $$\theta: \mathcal{Z}(\mathcal{U}(\mathbb{Z}[G])) \rightarrow \mathcal{Z}(\mathcal{U}(\mathbb{Z}[G])),\quad u\mapsto uu^{\ast},$$ is a group homomorphism and $\operatorname{ker}\theta = \pm \mathcal{Z}(G)$. Let $\mathcal{S}$ be the image of $\theta$. We will show that $\mathcal{S}$  is torsion free. Let $ u \in \mathcal{Z}(\mathcal{U}(\mathbb{Z}[G]))$.  If $uu^{\ast}$ is  of finite order, then it belongs to $\pm \mathcal{Z}(G)$. Suppose $$uu^{\ast}= \pm h,$$ where $h \in \mathcal{Z}(G)$. If  $h \neq 1$, then from the above equation it follows that 
	$\sum_{g \in G} u_{g}^{2}= 0$, which is not so as $ u \neq  0$.  Hence, $ h =1$ and  $uu^{\ast} = \pm 1$. However, $uu^{\ast}$ can't be $-1$ as the augmentation of $uu^{\ast}$ is $1$. Thus $uu^{\ast} =1$ and consequently, $\mathcal{S}$ is torsion free. 
	\vspace{.2cm} \\(ii) As $u$ is central in $\mathbb{Z}[G]$, it can be readily verified that $\theta(u^{2}) = \theta(uu^{\ast})$. Hence, by (i), $ u^2=\pm guu^{\ast}$, where $g \in \mathcal{Z}(G)$.  Since both $u^2$ and $uu^\ast$ are of augmentation $1$, we have $ u^2= guu^{\ast}$.  Hence, (ii) follows. \vspace{.2cm}\\
	(iii) If symmetric central units are  trivial, then from (ii) it follows that $u^{2} \in  \mathcal{Z}(G)$ for all  $ \mathcal{Z}(\mathcal{U}(\mathbb{Z}[G]))$. However, by (i),  $\mathcal{Z}(\mathcal{U}(\mathbb{Z}[G]))/\pm \mathcal{Z}(G)$ is torsion free. Therefore, (iii) follows. 
\end{proof}	

In the above theorem, (ii) may be campared with (\cite{MS99}, Theorem 1). An immediate consequence of the above theorem is the following:
\begin{cor}
	For every group $G$ with periodic center, in particular if $G$ is a finite group,  $\mathcal Z(\mathcal U(\mathbb Z[G]))/\mathcal Z_\mathcal S(\mathcal U(\mathbb Z[G])) $ is a torsion group, where $\mathcal Z_\mathcal S(\mathcal U(\mathbb Z[G])) $ is the subgroup of $\mathcal Z(\mathcal U(\mathbb Z[G]))$ consisting of symmetric central units. \end{cor}

\section{Hypercentral units } 
Let $\mathcal{Z}_{0}(\mathcal{U})\leq \mathcal{Z}_{1}(\mathcal{U})\leq \cdots \leq \mathcal{Z}_{n}(\mathcal{U})\leq \cdots $ be the upper central series of $ \mathcal{U}= \mathcal{U}(\mathbb{Z}[G])$. Let $\mathcal{Z}_{\infty}(\mathcal{U})= \cup_{i \geq 1} \mathcal{Z}_{i}(\mathcal{U})$ be the subgroup of $\mathcal{U}(\mathbb{Z}[G])$ consisting of the hypercentral units. 
\begin{theorem}\label{T0} For all $i \geq 2$, $\mathcal{Z}_{i}(\mathcal{U})/\mathcal{Z}_{i-1}(\mathcal{U})$ is of finite exponent, provided  $\mathcal{Z}(G)$ is of finite exponent or $G$ is generated by torsion elements of bounded exponent.  \end{theorem}
\begin{proof} In view of (\cite{Hal88}, Lemma 4.2, p. 432), it suffices to show that $\mathcal{Z}_{2}(\mathcal{U})/\mathcal{Z}_{1}(\mathcal{U})$ is of finite exponent.  Let $u \in \mathcal{Z}_{2}(\mathcal{U})$.  Suppose $\mathcal{Z}(G)$ is of finite exponent, say $m$.  Consider an arbitrary $g \in G$.  By (\cite{HIJJ07}, Proposition 4.1),  $[u, g] \in \mathcal{Z}(G)$.  Therefore,  $[u^{m}, g] = [u,g]^{m} =1$, i.e., $u^{m} \in \mathcal{Z}_{1}(\mathcal{U})$, as desired.  Next, suppose that $G$ is generated by $X$ and the elements of $X$ have bounded exponent, say $m$.  Then, for any $ x \in X$,  $[u,x] \in \mathcal{Z}_{1}(\mathcal{U})$ implies that $[u^{m}, x] = [u, x]^{m} =[u, x^{m}]=1$. Consequently $u^{m}$ is a central unit, as desired. This proves the result. 
 \end{proof} 
\begin{cor}\label{T2} If $G$ is such that all units in $\mathbb{Z}[G]$ are hypercentral and one of the following conditions hold: \begin{description}
	\item[(i)] $\mathcal{Z}(G)$ is of finite exponent; \item [(ii)] $G$ is generated by torsion elements of bounded exponent, \end{description}  then $\mathcal{U}(\mathbb{Z}[G])$ can't contain a free subgroup of rank $\geq 2$. \end{cor} [For the classification of groups $G$ with all units hypercentral, see \cite{IJ08}.]
\begin{proof} Suppose $F$ is a free subgroup of rank $\geq 2$ contained in $\mathcal{U}(\mathbb{Z}G)$. Consider any $1 \neq u \in F$. As $u$ is hypercentral, by Theorem \ref{T0}, it follows that a power of $u$ belongs to $\mathcal{Z}_{1}(\mathcal{U})$; let this power be $m$. So $u^{m}$ commutes with all elements of $F$, i.e., it belongs to $\mathcal{Z}(F)$. But $F$ being free of rank $\geq 2$, has trivial center. Consequently $u^{m}=1$. This is not possible, as $F$ is free.   \end{proof}

\par\vspace{.5cm}\centerline{\bf Acknowledgement}\par\vspace{.25cm}
Inder Bir S. Passi is thankful to the Indian National Science Academy, New Delhi (India) for their support and to Ashoka University, Sonipat (India), for making available their facilities.
 
\providecommand{\bysame}{\leavevmode\hbox to3em{\hrulefill}\thinspace}
\providecommand{\MR}{\relax\ifhmode\unskip\space\fi MR }
\providecommand{\MRhref}[2]{%
	\href{http://www.ams.org/mathscinet-getitem?mr=#1}{#2}
}
\providecommand{\href}[2]{#2}

\end{document}